\documentclass[12pt]{amsart}

\usepackage{amsthm, amsmath}
\usepackage{amssymb}
\usepackage[all]{xy}
\usepackage{mathrsfs}
\usepackage[latin1]{inputenc}
\usepackage{graphicx}

\numberwithin{equation}{subsection}

\DeclareMathOperator{\End}{End}

\theoremstyle{plain}

\newtheorem{theorem}{Theorem}[subsection]

\newtheorem{lemma}[theorem]{Lemma}

\theoremstyle{definition}
\newtheorem{rem}[theorem]{Remark}

\def\ge{\geqslant}
\def\le{\leqslant}
\def\a{\alpha}
\def\b{\beta}
\def\g{\gamma}

\def\d{\delta}

\def\L{\Lambda}
\def\e{\epsilon}

\def\s{\sigma}

\def\th{\theta}

\def\l{\lambda}

\def\i{^{-1}}

\def\ZZ{\mathbb Z}

\def\QQ{\mathbb Q}
\def\RR{\mathbb R}

\def\FF{\mathbb F}

\def\ci{\mathcal I}

\def\cu{\mathcal U}

\def\tW{\tilde W}
\def\tw{\tilde w}

\def\ba{\mathbf a}

\usepackage{amssymb} 
\usepackage{latexsym} 
\usepackage{amsfonts} 
\usepackage{amsmath} 
\usepackage{eucal} 
\usepackage{bm} 
\usepackage{bbm} 
\usepackage{graphicx} 
\usepackage[english]{varioref} 
\usepackage[nice]{nicefrac} 
\usepackage[all]{xy}

\newcommand{\kk}{\Bbbk}

\def\<{\langle}
\def\>{\rangle}
\def\tphi{\tilde \phi}
\def\Hom{\rm{Hom}}
\def\tPhi{\tilde \Phi}
\def\te{\tilde e}
\def\h{\rm{h}}
\def\bc{\bar c}
\def\ba{\bar a}

\begin{document}

\title[]{On isomorphism numbers of ``$F$-crystals"}
\author{Sian Nie}
\address{Max-Planck Institute for Mathematics, Vivatsgasse 7, 53111 Bonn}
\email{niesian@amss.ac.cn}

\begin{abstract}
In this note, we show that for an ``$F$-crystal" (the equal characteristic analogue of $F$-crystals), its {\it isomorphism number} and its {\it level torsion} coincide. This confirms a conjure of Vasiu \cite{Va} in the equal characteristic case.
\end{abstract}

\maketitle
\section{Introduction}
\subsection{} Let $\kk$ be an algebraically closed field of characteristic $p>0$ and $D$ a $p$-divisible group over $\kk$. It is well-known that there exists a nonnegative integer $n$ such that the isomorphic class of $D$ is determined by the finite truncation $D[p^n]$ of $D$. We call the smallest integer with such property the {\it isomorphism number} of $D$ and denote it by $n_D$. An interesting problem is to compute this number using invariants of the isomorphism class of $D$. In \cite{Va}, Vasiu introduced the {\it level torsion $\ell_D$} of $D$ by considering its associated Dieudonn\'{e} module and showed that $n_D \le \ell_D$. Actually Vasiu proved this inequality for all {\it $F$-crystals} and moreover, the equality hold if the $F$-crystal is a direct sum of {\it isoclinic} ones. He also conjectured that the equality holds for all $F$-crystals. In \cite{LNV}, Lau, Nicole and Vasiu discovered an optimal upper bound for $n_D$ in terms of the Hodge polygon and the Newton polygon of $D$. They also proved that the equality $n_D=\ell_D$ always holds. Recently, using level torsion numbers, Xiao \cite{X1} provided an efficiently computable upper bound for isomorphism numbers of isoclinic $F$-crystals.

The main purpose of this paper is to prove Vasiu' conjecture on the equality of the isomorphic number and level torsion for an {\it``$F$-crystal"} (see \S\ref{def}), defined using the ring $\kk[[\e]]$ of formal power series instead of the ring $W(\kk)$ of Witt vectors over $\kk$.

\subsection{}\label{def} Let $\kk((\e))$ be the field of formal Laurent series over $k$ and $\kk[[\e]]$ its ring of formal power series. Let $q$ be a power of $p$. We define the Frobenius automorphism $\s$ of $\kk((\e))$ by $\s(\sum a_n \e^n) =\sum a_n^q \e^n$, where each $a_n \in \kk$.

Let $M$ be a free $\kk[[\e]]$-module of finite rank. Let $\phi$ be a $\s$-linear isomorphism of $V=M \otimes_{\kk[[\e]]} \kk((\e))$. We call the pair $(V, \phi)$ an {\it ``$F$-isocrystal"} over $\kk$. Moreover, if $\phi(M) \subset M$, we call the pair $(M, \phi)$ an {\it ``$F$-crystal"} over $\kk$. Its {\it isomorphism number}, denoted by $n_M$, is the smallest nonnegative integer $n$ satisfying the property that for any $h \in GL(M)$ with $h \equiv 1 \mod \e^n$, we have $(M, h \phi)$ is isomorphic to $(M, \phi)$, that is, $h \phi=g\i \phi g$ for some $g \in GL(M)$.

In the above definitions, if we replace $\kk[[\e]]$ and $\kk((\e))$ by the ring of Witt vectors over $\kk$  and its fractional field respectively, then we obtain the notions of $F$-isocrystals and $F$-crystals.

Let $\tphi$ be the $\s$-linear automorphism of $\End(V)$ defined by $f \mapsto \phi f \phi\i$. We see that $(\End(V), \tphi)$ is an ``$F$-isocrystal".

By Dieudonn\'{e}'s classification on {\it ``$F$-isocrystals"}, we have a decomposition $V=\oplus_{\l} V^{\l}$, where each $V^{\l}$ is a $\phi$-stable $\kk((\e))$-vector subspace such that $(V^{\l}, \phi|_{V^{\l}})$ is an ``$F$-isocrystal" of {\it slope} $\l \in \QQ$. Then $\End(V)=V_+ \oplus V_0 \oplus V_-$, where $V_+=\bigoplus_{\l_1 < \l_2} \Hom(V^{\l_1}, V^{\l_2})$, $V_0=\bigoplus_{\l} \Hom(V^{\l}, V^{\l})$ and $V_-=\bigoplus_{\l_1 < \l_2} \Hom(V^{\l_2}, V^{\l_1})$.

Let $H$ be a lattice of $\End(V)$. Let $O_{H,+} \subset V_+ \cap H$, $O_{H,0} \subset V_0 \cap H$ and $O_{H,-} \subset V_- \cap H$ be the maximal $\kk[[\e]]$-modules  such that $\tphi(O_{H,+}) \subset O_{H,+}$, $\tphi(O_{H,0}) = O_{H,0}$ and $\tphi\i(O_{H,-}) \subset O_{H,-}$. Define $$O_H=O_{H,+} \oplus O_{H,0} \oplus O_{H,-}.$$

Following \cite{Va}, if $H=\End(M)$, we call $O_H$ the {\it level module} of $(M, \phi)$ and define the {\it level torsion} $\ell_M$ of $(M, \phi)$ by the following two rules:

(a) $\ell_M=1$ if $H=\End(M)=O_H$ and $O_{H,+} \oplus O_{H,-}$ is not topologically nilpotent;

(b) $\ell_M=\min\{n \in \ZZ_{\ge 0}; \e^n H \subset O_H\}$ otherwise.

\subsection{} The main result of this note is the following.
\begin{theorem}\label{main}
For any ``$F$-crystal" $(M, \phi)$ we have $n_M \ge \ell_M$.
\end{theorem}

Combining with Vasiu's inequality $n_M \le \ell_M$ in \cite{Va}, we confirm Vasiu's conjecture for ``$F$-crystals".
\begin{theorem}\label{main'}
For any ``$F$-crystal" $(M, \phi)$ we have $n_M = \ell_M$.
\end{theorem}

\begin{rem}
Vasiu only proved his inequality for $F$-crystals. However, his proof can be adapted for ``$F$-crystals" straightforwardly.
\end{rem}



After a preliminary version of this note was finished, we learnt that Xiao has confirmed Vasiu's conjecture for $F$-crystals in the preprint \cite{X2}, where he systematically generalizes several invariants, originally defined for isomorphism classes of $p$-divisible groups, to those of $F$-crystals. Then using a main strategy in \cite{LNV} he shows that all these invariants, including level torsion and isomorphism number, are all the same.

In this note, we approach Theorem \ref{main} in a different way. The proof consists of two steps. First we prove a ``linearization" lemma.
\begin{lemma}\label{first}
If $K_n \subset \{g\i \tphi(g); g \in K\}$ for some $n \in \ZZ_{>0}$, then $\e^n \End(M) \subset (\tphi-1) \End(M)$. Here $K_n=\{g \in GL(M); g \equiv 1 \mod \e^n\}$.
\end{lemma}

Then we show
\begin{lemma}\label{sec}
If $\e^n H \subset (\tphi-1) H$ for some $n \in \ZZ_{\ge 0}$, then $\e^n H \subset O_H$. Recall that $H$ is a lattice of $\End(M)$.
\end{lemma}

Assuming these two lemmas, one can finish the proof immediately.
\begin{proof}[Proof of Theorem \ref{main}]
If $n_M=0$, it follows form \cite[Lemma 2.3]{Va}. Now we assume $n_M \ge 1$ and it follows directly from Lemma \ref{first} and Lemma \ref{sec}.
\end{proof}

\

{\it Acknowledgement.} I would like to thank X. Xiao for sending me his preprint \cite{X2} and his helpful correspondence.

\section{Algebraically closed admissible subsets}
Let $U$ be a finite dimensional $\kk((\e))$-vector space and $N \subset U$ a lattice. For an integer $i \in \ZZ$ we set $N_i=\e^i N$. Let $\psi$ be a $\s$-linear automorphism of $U$. We say $Z \subset U$ is an algebraically {\it closed admissible subset} if there exist $n < m \in \ZZ$ and a closed subvariety $Z_{n,m} \subset N_n / N_m$ such that $Z$ is the preimage of $Z_{n,m}$ under the natural projection $\pi_{n,m}: N_n \to N_n / N_m$. Clearly our definition does not depend on the choices of $n$, $m$ and $Z_{n,m}$. The following lemma implies that it doesn't depend on the choice $N$ either.
\begin{lemma}\label{linear}
$GL(U)$ preserves algebraically closed admissible subsets of $U$.
\end{lemma}
\begin{proof}
We fix a $\kk[[\e]]$-basis of $E=\{e_i; 1 \le i \le n\}$ $N$. By Cartan decomposition we have $GL(U)=K Y K$, where $K=GL(N)$ and $$Y=\{g \in GL(U); g(e_i)=\e^{a_i}e_i \text{ for some } a_i \in \ZZ, 1 \le i \le n\}.$$ Hence it suffices to show that $K \cup Y$ preserves algebraically closed admissible subsets of $U$. Let $n < m \in \ZZ$ and $g \in K \cup Y$. Then $g(N_n)$ is an algebraically closed admissible subset. Note that the natural morphism $g: N_n / N_m \to g(N_n) / g(N_m)$ induced by $v \mapsto g(v)$ for $v \in N_n$ is an isomorphism. Then the lemma follows from the observation that the natural projection $g(N_n) / N_{m'} \to g(N_n) / g(N_m)$ is a trivial vector bundle for any sufficiently large $m' \in \ZZ$.
\end{proof}

Now we come to the main result of this section, which plays a key role in the proof of Lemma \ref{first}.
\begin{lemma}\label{closed}
Let $\psi$ be a $\s$-linear automorphism of $U$. Then $\psi-1$ preserves algebraically closed admissible subsets of $U$.
\end{lemma}
\begin{proof}[Proof of Proposition \ref{closed}]
Let $U=\bigoplus_{\l} U^{\l}$ be the slope decomposition with respect to $\psi$. We fix a basis $E^\l=\{e_{i, \l}; 1 \le i \le r_\l\}$ for each $U^{\l}$. Thanks to Lemma \ref{linear}, we may conjugate $\psi$ by a proper element of $GL(U)$ so that $\psi(e_{i,\l})=e_{i+1,\l}$ for $i=1, \cdots, r_\l-1$ and $\psi(e_{r_\l,\l})=\e^{s_\l} e_{1,\l}$, where $s_\l=\l r_\l \in \ZZ$. Moreover, if $\l=0$, we can assume $\psi(e_i)=e_i$ for each $e_i \in E^\l$.

Let $N^\l$ be the lattice spanned by $E^\l$ and $N^\l_n=\e^n N^\l$ for any $n \in \ZZ$. We see that $(\psi-1) N^\l =\max \{N^\l, \psi(N^\l)\}=N^\l+\psi(N^\l)$.

Let $n < m \in \ZZ$. To prove the lemma, it suffices to show the natural surjective morphism
\begin{align}\label{fin} \psi-1: \frac{N^\l_n}{N^\l_m} \longrightarrow \frac{(\psi-1)N^\l_n}{ (\psi-1)N^\l_m}
\end{align} is a finite morphism.

Without loss of generality, we assume $\l < 0$. Then (\ref{fin}) becomes $$\psi-1: \frac{N^\l_n}{N^\l_m} \longrightarrow \frac{\psi(N^\l_n)}{\psi(N^\l_m)}.$$ Since $N^\l_n/N^\l_m$ is isomorphic to the $\kk$-vector space spanned by $\e^j e_{i,\l}$ for $1 \le i \le r_\l$ and $n \le j \le m-1$, its coordinate ring is $\kk[N^\l_n/N^\l_m]=\kk[x_{i,j}; 1 \le i \le r_\l, n \le j \le m-1]$, where $x_{i,j}$ is the dual linear function such that $x_{i,j}(\e^{j'}e_{i',\l})=\d_{i,i'}\d_{j,j'}$. Similarly $\kk[\psi(N^\l_n)/\psi(N^\l_m)]=\kk[y_{i,j}; 1 \le i \le r_\l, n \le j \le m-1]$, where $y_{i,j}(\psi(\e^{j'} e_{i',\l}))=\d_{i,i'}\d_{j,j'}$. Let $\psi^*: \kk[\psi(N^\l_n)/\psi(N^\l_m)] \to \kk[N^\l_n/N^\l_m]$ be the induced homomorphism between coordinate rings. One computes directly that \begin{align*}(\psi^*-1)(y_{i,j})=\begin{cases}& x_{i,j}^q-x_{i+1,j}, \text{ if } 1 \le i \le r_\l-1;\\ & x_{r_\l,j}^q-x_{1,j-s_\l}, \text{ if } i=s_\l, j-s_\l \ge n; \\ & x_{r_\l,j}^q, \text{ otherwise. }\end{cases}\end{align*} Let $l \in \ZZ$ such that $n \le l \le n+s_\l-1$. By the above rules, $x_{r_\l,l}^q=(\psi^* -1)(y_{r_\l,l})$. Hence $x_{r_\l,l}$ is integral over $\kk[\psi(N^\l_n)/\psi(N^\l_m)]$. Note that $(\psi^*-1)(y_{r_\l-1,l})=x_{r_\l-1,l}^q-x_{r_\l,l}$. So $x_{r_\l-1,l}$ is integral over $\kk[\psi(N^\l_n)/\psi(N^\l_m)]$. Repeating this argument, we see that $x_{i,j}$ is integral over $\kk[\psi(N^\l_n)/\psi(N^\l_m)]$ for $1 \le i \le r_\l$ and  $n \le j \le m-1$. Therefore $\psi-1$ is a finite morphism as desired.
\end{proof}

\section{Proof of Lemma \ref{first}} We keep $M$, $\phi$ and $V$ as in the section of introduction. Let $H=\End(M)$ and $K=GL(M)$. For each $n \in \ZZ_{\ge 0}$, we define $K_n=\{g \in GL(M); g \equiv 1 \mod \e^n\}$. The main purpose of this section is to prove Lemma \ref{first}. The proof is based on {\it Bruhat decomposition} of $GL(V)$ and certain {\it Moy-Prasad filtration} of $\End(M)$, which we now introduce.

Fix a basis $E=\{e_i; 1 \le i \le r\}$ of $M$ and identify $K$ with $GL_r(\kk[[\e]])$ so that for each $g=(g_{i,j}) \in GL_r(\kk[[\e]])=K$, we have $g(e_j)=\sum_{i=1}^r g_{i,j}e_i$. Let $M^\vee$ be the dual lattice of $M$ and let $E^\vee=\{e_i^\vee; 1 \le i \le r\} \subset M$ be the basis dual to $E$. Then $H=M \otimes_{\kk[[\e]]} M^\vee$. Let $\bar M \subset M$ ($\bar M^\vee \subset M^\vee$) be the $\kk$-vector space spanned by $E$ ($E^\vee$). Let $\bar H=\bar M \otimes \bar M^\vee$ and $\bar K=GL(\bar M) \cong GL_r(\kk) \subset K$.

Let $B \subset \bar K$ be the subgroup of upper-triangulated matrices. Let $U \subset B$ be its unipotent radical and $T \subset B$ the subgroup of invertible diagonal matrices. Then $B=TU$. Let $\pi: K \to \bar K$ be the reduction modulo $\e$ map. Let $\ci=\pi\i(B)$, which is called a {\it standard Iwahori subgroup} with respect to $E$. Let $\cu=\pi\i(U)$ be the {\it prounipotent subgroup} of $\ci$. Let $W \subset \bar K$ be the subgroup of permutation matrices. Let $Y \cong \ZZ^r$ be the cocharacter group of $T$. Since $W$ normalizes $T$, we define $\tW=Y \rtimes W$. Then we have the Bruhat decomposition $$GL(V)=\bigsqcup_{\tw \in \tW} \ci \tw \ci.$$

Let $\Phi=\{\a_{i,j}=e_i-e_j; 1 \le i, j \le r\} \setminus \{0\}$ be the set of roots and $\tPhi=\{\a+n \in \ZZ \Phi \oplus \ZZ; \a \in \Phi \cup \{0\}, n \in \ZZ\} \setminus \{0\}$ the set of affine roots. Set $ \tPhi^+=\{\a_{i,j}+n \in \tPhi; \text{ either $n>0$ or $n=0$ and $i < j$}\}$. Then $\tPhi=\tPhi^+ \cup -\tPhi^+$. Note that $\tW$ acts on $\tPhi$. We define the {\it length} of $\tw \in \tW$ by $\ell(w)=\sharp \{a \in \tPhi^+; \tw(a)<0\}$. Let $S=\{s \in W; \ell(s)=1\}$ be the set of simple reflections of $W$. Let $s \in S$ and $\tw \in \tW$. We have \begin{align*}\ci \tw \ci s \ci=\begin{cases}& \ci \tw s \ci, \text{ if } \ell(\tw s)=\ell(\tw)+1; \\ & \ci \tw s \ci \bigsqcup \ci \tw \ci, \text{ otherwise. } \end{cases} \end{align*} Let $\tW {}^S=\{\tw \in \tW; \ell(ws)=\ell(w)+1, s \in S\}$. For any $x \in \tW {}^S$ and $w \in W$ we have $\ell(xw)=\ell(x)+\ell(w)$ and $\ci x \ci w \ci= \ci x w \ci$.

Let $a=\a_{i,j}+n \in \tPhi$. We define $\te_a=\e^n e_i \otimes e_j^\vee \in H$ and $U_a=\{1+ c \te_a \in GL(V); c \in \kk\}$. Then $\cu=\prod_{a \in \tPhi^+} U_a$ and each $g \in G(L)$ has a unique expression $g=h w i$ with $\tw \in \tW$, $i \in \ci$ and $h \in U_w \doteq \prod_{\a \in \tPhi, \tw\i(\a)<0} U_{\a}$.

Let $J \subset S$. Let $P_J=L_J U_J \supset T$ be the standard parabolic subgroup of type $J$ with Levi subgroup $L_J \supset T$ and unipotent radical $U_J$. Let $\Phi_J \subset \Phi$ be the set of root of $L_J$. Then $L_J$ is generated by $T$ and the root subgroup $U_\a$ with $\a \in \Phi_J$. Let $W_J \subset W$ be the Weyl group of $L_J$. We have $L_J=\bigsqcup_{w \in W_J} B_J w B_J$, where $B_J=B \cap L_J$ is a Borel subgroup of $L_J$. Let $\cu_J=\pi\i(U_J)=U_J K_1 \subset \cu$. Then $\cu_J$ is normalized by $L_J$.

For $g, f \in GL(V)$ and $D \subset GL(V)$ we write ${}^g f=g f g\i$ and ${}^g D=\{{}^g h; h \in D\}$. By abuse of notations, we denote by $\s: V \to V$ the $\s$-linear endomorphism of $V$ which fixes each element in the basis $E$ of $M$.

\begin{lemma}\label{f}
Let $\phi \in K_1 x \s$ be a $\s$-linear automorphism of $V$ with $x \in \tW {}^S$. Let $J=\max\{J' \subset S; x J' x\i=J'\}$. If $K_n \subset \{g\i \tphi(g); g \in K\}$ for some $n \in \ZZ_{>0}$, then $K_n \subset \{g\i \tphi(g); g \in \L \cu_J\}$, where $\L=\{g \in L_J; {}^{x\s} g=g\}$ is a finite set.
\end{lemma}
\begin{proof}
Let $g \in K$ such that $f=g\i \tphi(g) \in K_n \subset K_1$. Assume $g=h w i$ with $w \in W$, $h \in U_w \subset \bar K$ and $i \in \ci$. By $f \phi=g\i\phi g$ we have $$\ci x\s \ci = \ci w\i h\i x\s\ h w \ci.$$ Note that ${}^\s w=w$ and $\ell(xw)=\ell(x)+\ell(w)$. We have $w\i x w=x$. By \cite[Lemma 4.1]{N}, $w \in W_J$ and $g \in L_J \cu_J$. Write $g=m u$ with $m \in L_J$ and $u \in \cu_J$. Then ${}^{x\s}u \in \cu$ and $$K_1 x \s=K_1 u\i m\i {}^{x\s}m\ {}^{x\s}u\ x\s \subset \cu_J L_J x\s $$ Since $x J x\i=J$, $m'=m\i {}^{x\s}m \in L_J$. Assume $m' \in B_J w' B_J$ for some $w' \in W_J$. Then $x\s \in \ci w' \ci x\s \ci=\ci w'x\s \ci$. So $w'=1$ and $m' \in B_J$. Write $u=\prod_{a \in \tPhi^+ \setminus \Phi_J} x_a$ with $x_a \in U_a$. Since ${}^{x\s}u \in \cu$, $x\s(a)>0$ if $x_a \neq 1$. Noticing that ${}^{x\s} B_J=B_J$, we have furthermore $x\s(a) \in \tPhi^+ \setminus \Phi_J$ if $x_a \neq 1$, which means ${}^{x\s}u \subset \cu_J$ and $m'=1$. Therefore $g \in \L \cu_J=\cu_J \L$ as desired.

It remains to show $\L$ is finite. Indeed, we have $(x\s)^s=\mu(\e) \s^s$ for some positive integer $s$ and some cocharacter $\mu \in Y$, which has only finitely many fixed points on $L_J$. Therefore $\L$ is finite.
\end{proof}

\begin{proof}[Proof of Lemma \ref{first}]
By \cite{Vi}, conjugating $\phi$ by a proper element of $K$ we may assume $\phi \in K_1 x\s$ with $x \in \tW {}^S$. Thanks to Lemma \ref{f}, there exists $J \subset S$ with $x J x\i=J$ such that $K_n \subset \{g\i \tphi(g); g \in \L \cu_J\}$, where $\L=\{g \in L_J; {}^{x\s}g=g\}$.

Fix $v \in Y \otimes_{\ZZ} \RR$ such that $\<\a, v\>=0$ if $\a \in \Phi_J$ and $0<\<\a, v\> < \frac{1}{2}$ if $\a \in \Phi^+ - \Phi_J$. Let $r \in \RR$. Define $\tPhi_r=\{a \in \tPhi; a(v)= r\}$ and $\tPhi_{\ge r}=\{a \in \tPhi; a(v)\ge r\}$. Set $V\<r\>=\oplus_{a \in \tPhi_r} \kk \te_a$ and $H_{\ge r}=\sum_{a \in \tPhi_{\ge r}} \kk[[\e]] \te_a$. Note that $H_{\ge r} H_{\ge r'} \subset H_{\ge r+r'}$ for $r, r' \in \RR$.

We prove the lemma by induction on $0 \le r \in \RR$ such that $H_{\ge r} \subset (\tphi-1)H$.  By Lemma \ref{closed}, there exists a sufficiently large integer $m$ such that $H_m \subset (\tphi-1) H$. We can assume $H_{\ge r} \subset H_n$ and $H_{\ge r} \subset (\tphi-1) H$ for some $r \ge 0$. If $H_n=H_{\ge r}$, there is nothing to prove. Now we assume that $H_{\ge r} \subsetneq H_n$. Let $r^-=\max \{t \in \RR; t < r, \tPhi_{\ge r} \subsetneq \tPhi_{\ge t}\}$. We show that $H_{\ge r^-} \subset H_n$ and $H_{\ge r^-} \subset (\tphi-1)H$, which will finish our induction argument.

Assume $H_{r^-} \nsubseteq H_n$. Since $H_{\ge r^-}=V\<r^-\> \oplus H_{\ge r}$, there exists $a=\a+j \in \tPhi_{r^-}$ with $\a \in \Phi \cup \{0\}$ and $j \in \ZZ$ such that $\te_a \in V\<r^-\> \setminus H_n$, that is, $a(v)=\<\a, v\>+j=r^-$ and $j \le n-1$. Let $b=\b+i \in \tPhi$ with $\b \in \Phi \cup \{0\}$ and $i \in \ZZ$ such that $\te_b \in H_n$. Then $i \ge n$ and \begin{align*}b(v) &=a(v)+(i-j)+\<\b-\a, v\> \\ &\ge a(v)+(i-j)-|\<\b-\a,v\>| \\ &> a(v)+1-2 \times \frac{1}{2}=r^-.\end{align*} By the choice of $r^-$, we have $b(v) \ge r$, hence $b \in \tPhi_{\ge r}$ which means $H_n \subset H_{\ge r}$ contradicting our assumption that $H_n \nsubseteq H_{\ge r}$. Therefore $H_{\ge r^-} \subset H_n$.

It remains to show $H_{\ge r^-} \subset (\tphi-1)H$. Let $u \in V\<r^-\> \subset H_n$. Then $1+u \in K_n$. By assumption, there exists $g \in \L \cu_J$ such that $$g u=(\tphi-1)(g).$$ Write $g=(1+f)h$ with $h \in \L$ and $1+f \in \cu_J$. We have $$h u +f h u=(\tphi-1)(g).$$
Since $h \in L_J$ and $1+f \in \cu_J$, we have $h V\<r^-\>=V\<r^-\>$ and $f V\<r^-\> \subset H_{\ge r} \subset (\tphi-1) H$ by the choice of $r^-$. Hence
\begin{align*}\label{*}
\tag{$\ast$} h u \in V\<r^-\> \cap (\tphi-1)H.
\end{align*}
Applying Lemma \ref{closed} with $U$ and $\psi$ replaced by $\End(V)$ and $\tphi$ respectively, we see that $H_{\ge r^-} \cap (\tphi-1) H \supset H_{\ge r}$ is an algebraically closed admissible set. Hence the image $X$ of the natural projection $$H_{\ge r^-} \cap (\tphi-1) H \to H_{\ge r^-} / H_{\ge r}=V\<r^-\>$$ is a closed subvariety. By (\ref{*}) we see that $$V\<r^-\> \subset \bigcup_{h \in \L} h\i X.$$ Since $\L$ is finite and that $V\<r^-\>$ is an irreducible variety, we see that $X=V\<r^-\>$ and $H_{\ge r^-} \subset (\tphi-1)H$ as desired.
\end{proof}

\section{Proof of Lemma \ref{sec}}
In this section, we prove Lemma \ref{sec} by reduction to absurdity.

Let $V$, $\phi$ and $H \subset \End(V)$ be the same as in the introduction section. Let $A_0=\{f \in \End(V); \tphi(f)=f\}$. Then $A_0$ is a $\FF_q((\e))$-vector space and $O_{H,0}=(A_0 \cap H) \otimes_{\FF_q[[\e]]} \kk[[\e]]$. We denote by $v_H: \End(V) \to \ZZ$ the associated valuation such that $v_H(f)=-\inf\{k \in \ZZ; \e^k f \in H\}$. Let $v: \kk((\e)) \to \ZZ$ be the ordinary valuation such that $v(\kk)=0$ and $v(\e)=1$.

\begin{proof}[Proof of Lemma \ref{sec}]
Let $h \in \e^n H$. We have to show that $h \in O_H$. Assume $h \in \e^m O_H$ for some sufficiently small integer $m$. Let $E'=\{e_i'; i \in I\}$ be a $\FF_q[[\e]]$-basis of $A_0 \cap \e^m H$.

Write $h=x_+ + x_0 + x_-$ with $x_- \in V_-$, $x_0 \in V_0$ and $x_+ \in V_+$. Recall that $V_0$, $V_\pm$ are defined in \S\ref{def}. Write $x_0= \sum_{i \in I}c_i^o e_i^o$ with each $c_i^o \in \kk[[\e]]$. Define $$n_0=\min\{v_H(c_i^o e_i^o), v_H(\tphi^k(x_+), v_H(\tphi^{-k-1}(x_-)); i \in I, k \in \ZZ_{\ge 0}\}.$$ since $\tphi(\e)=\e$, we have $v_H(e_i^o)=m$ and $v_H(c_i^o e_i^o)-n_0=v(c_i^o)+m- n_0 \ge 0$. Hence $c_i^o e_i^o=\e^{m-n_0} c_i^o (\e^{n_0-m}e_i^o) \in \e^{n_0} O_{H,0}$, which implies $x_0 \in \e^{n_0} O_{H,0}$. By definition, $x_+ \in \e^{n_0} O_{H,+}$. If $n_0 \ge 0$, then $x_-=h-x_+-x_0 \in H$ and hence $x_- \in O_{H,-}$ by definition. Therefore $h \in O_H$ as desired.

Now we assume $n_0 \le -1$. We show that this would lead to a contradiction. Note that $x_- \in \e^{n_0} H$ and hence $x_- \in \e^{n_0} O_{H,-}$. Therefore $h \in \e^{n_0} O_H$.

For a basis $E'=\{e_i'; i \in I\}$ of $A_0 \cap \e^{n_0} H$, we denote by $\h_{E''}(x_0)$ the numbers of nonzero coefficients $c_i' \in \kk[[\e]] \setminus \e \kk[[\e]]$ in the expression $x_0=\sum c_i' e_i'$. Let $E=\{e_i; i \in I\}$ be a basis of $A_0 \cap \e^{n_0} H$ such that $\h_E(x_0)$ is minimal among all bases of $A_0 \cap \e^{n_0} H$. Let $\bc_i \in k$ such that $c_i-\bc_i \in \e \kk[[\e]]$ and $J=\{i \in I; \bc_i \neq 0\}$. Then $\h_E(x_0)=|J|$. We show that the elements $\bc_i$ for $i \in J$ are linear independent over $\FF_q$. Otherwise, there exists $j' \in J$ such that $\bc_{j'}=\sum_{j \in J \setminus \{j'\}} b_j \bc_j$ with each $b_j \in \FF_q$. Take $e_i^1=e_i+b_i e_{j'}$ if $i \in J \setminus \{i'\}$ and $e_i^1=e_i$ otherwise. Then $E^1=\{e_i^1; i \in I\}$ is a basis of $A_0 \cap \e^{n_0} H$ such that $\h_{E''}(x_0)=\h_E(x_0)-1$, which contradicts our choice of $E$.

We claim that $$n_0=\min\{v_H(c_i e_i), v_H(\tphi^k(x_+), v_H(\tphi^{-k-1}(x_-)); i \in I, k \in \ZZ_{\ge 0}\}.$$ Indeed, let $E''=\{e_i''; i \in I\}$ be a basis of $A_0 \cap \e^{n_0} H$ and denote by $C_{E''}(x_0)$ the ideal of $\kk[[\e]]$ generated by the coefficients $c_i'' \in \kk[[\e]]$ in the expression  $x_0=\sum_{i \in I} c_i'' e_i''$. Then \begin{align*}\min\{v_H(c_i'' e_i''); i \in I\} &= n_0+\min\{v(c_i''); i \in I\} \\ &= n_0+\min\{v(c); c \in C_{E''}(x_0)\},\end{align*} where the first equality follows from that observation that $v_H(e_i'')=n_0$ for $i \in I$. Note that $C_{E''}(x_0)$ is independent of the choice of $E''$ and that $\{\e^{n_0-m} e_i^o; i \in I\}$ is a basis of $A_0 \cap \e^{n_0} H$. We have $\min\{v_H(c_i e_i); i \in I\}=\min\{v_H(c_i^o e_i^o)=v_H(\e^{m-n_0}c_i^o(\e^{n_0-m}e_i^o)); i \in I\}$ and the claim is proved.

We fix a basis $F=\{f_l; l\in \Pi\}$ of $H$ and Let $\bar H \subset H$ be the $\kk$-vector subspace spanned $F$. Let $u_k, v_i, w_k \in \bar H$ be the unique elements such that $$-\e^{-n_0} \tphi^k(x_+)-u_k,\ \e^{-n_0} e_i-v_i,\ \e^{-n_0} \tphi^{-k-1}(x_-)-w_k \in \e H$$ for $i \in I$ and $k \in \ZZ_{\ge 0}$. Let $\a_{k,l}, \b_{i,l}, \g_{k,l} \in \kk$ be such that $u_k=\sum \a_{k,l} f_l$, $v_i=\sum \b_{i,l} f_l$ and $w_k=\sum \g_{k,l} f_l$. Note that $\a_{l,k}$ and $\b_{l,k}$ are zero for all but finitely many $k \in \ZZ_{\ge 0}$.

Let $h' =x_+' + x_0' + x_-' \in \e^{n_0} O_H$, where $x_\pm' \in \e^{n_0} O_{H,\pm}$ and $x_0'=\sum_{i \in I}d_i' e_i \in \e^{n_0} O_H$ with each $c_i' \in \kk[[\e]]$. Following \cite[Lemma 8.6]{LNV} we define $y_{h',+}=-\sum_{k=0}^\infty \tphi^k(x_+')$, $y_{h',-}=\sum_{k=0}^\infty \tphi^{-k-1}(x_-')$ and $y_{h',\th',0}=\sum_{i \in I} \th_i' e_i$, where $\th=(\th_i)_{i \in I}$ with each $\th_i' \in \kk[[\e]]$ satisfying $\s(\th_i')-\th_i'=d_i'$. Set $y_{h',\th'}=y_{h',+}+y_{h',\th',0}+y_{h',-} \in \e^{n_0}O_H$. Then we have $h'=\tphi(y_{h',\th'})-y_{h',\th'}$.

Let $r \in \kk$. Choose $\th_r=(\th_{i,r})_{i \in I}$ with each $\th_{i,r} \in \kk[[\e]]$ satisfying $\s(\th_{i,r})-\th_{i,r}=rc_i$ for $i \in I$. By assumption, $rh=(\tphi-1)(y_{rh,\th})=(\tphi-1)(g_r)$ for some $g_r \in H$. Therefore \begin{align} \label{f0} y_{rh,\th_r}-g_r \in A_0 \cap \e^{n_0}H. \end{align} Hence there exist $a_{i,r} \in \FF_q[[\e]]$ for $i \in I$ such that \begin{align} \label{f1} y_{rh,\th_r}+\sum_{i \in I} a_{i,r} e_i \in H. \end{align}

Let $z_{i,r}=\bar \th_{i,r} \in \kk$. Then $z_{i,r}^q - z_{i,r}=r \bc_i$. By (\ref{f0}) and the assumption $n_0 \le -1$, we have $$\sum_{i \in I} (z_{i,r}+\ba_{i,r}) v_i +\sum_{k=0}^\infty (r^{q^k} u_k + r^{q^{-k-1}} w_k)=0.$$ In other words, for each $l \in \Pi$, we have \begin{align}\label{key}\sum_{i \in I}\b_{i,l} (z_{i,r}+\ba_{i,r})+\sum_{k=0}^\infty (\a_{k,l} r^{q^k} + \g_{k,l} r^{q^{-k-1}})=0.\end{align} We show that there exists $l_0 \in \Pi$ such that (\ref{key}) holds for only finitely many $r \in \kk$, which is a contradiction since $\kk$ is an infinite set.

Define $J=\{i \in I; \bc_i \neq 0\}$. If $J=\emptyset$, then $z_{i,r} \in \FF_q$ for all $i \in I$. Moreover, Since $c_i e_i \in \e^{n_0+1} H$ for all $i \in I$, we see that there exists $l_0 \in F$ such that one of $\{\a_{k,l_0}, \g_{k,l_0}; k \in \ZZ_{\ge 0}\}$ is nonzero. Now (\ref{key}) becomes
\begin{align}\label{key1}
\sum_{k=0}^\infty (\a_{k,l_0} r^{q^k} + \g_{k,l_0} r^{q^{-k-1}}) \in \sum_{i \in I} \FF_q \b_{i,l_0},
\end{align}
which has only finitely many solutions for $r$ as desired.

Now we assume $J \neq \emptyset$ and fix an element $i_0 \in J$. Since $v_i \neq 0$ for each $i \in I$, there exists $l_0 \in \Pi$ such that $\b_{i_0,l_0} \neq 0$. We write $\b_i=\b_{i,l_0}$, $\a_k=\a_{k,l_0}$ and $\g_k=\g_{k,l_0}$ for $i \in I$ and $k \in \ZZ_{\ge 0}$. Take $z_i=z_{i,r}+\ba_{i,r}$ for $i \in I$. Then (\ref{key}) is equivalent to the following equations with variables $z_i$ with $i \in I$.
\begin{align}
\label{key2'}   & P_i=\bc_i\i (z_i^q-z_i)-\bc_{i_0}\i (z_{i_0}^q-z_{i_0})=0 \text{ for } i \in J \setminus \{i_0\}; \\
\label{key2''}  & P_i=z_i^q-z_i=0 \text{ for } i \in I \setminus J; \\
\label{key2'''} & P_{i_0}=\sum_{i \in I}\b_i z_i+\sum_{j=0}^\infty (\a_j (\bc_{i_0}\i (z_{i_0}^q-z_{i_0}))^{q^j} + \g_j (\bc_{i_0}\i (z_{i_0}^q-z_{i_0}))^{q^{-j-1}})=0.
\end{align}
To obtain the contradiction, it suffices to show that the above three equations have only finitely many solutions for $z_{i_0}$.

If not all of $\g_k$ with $k \in \ZZ_{\ge 0}$ are zero, let $k_0$ be the maximal integer such that $\g_{k_0} \neq 0$. In this case, replacing $P_{i_0}$ by $P_{i_0}^{q^{k_0+1}}$ in (\ref{key2'''}) yields
\begin{align*}
P_{i_0}&= \sum_{k=0}^\infty (\g_k^{q^{k_0+1}} (\bc_{i_0}\i (z_{i_0}^q-z_{i_0}))^{q^{k_0-k}}+\a_k^{q^{k_0+1}}(\bc_{i_0}\i (z_{i_0}^q-z_{i_0}))^{q^{k+k_0+1}}) \\ & + \sum_{i \in I}\b_i^{q^{k_0+1}} z_i^{q^{k_0+1}}=0.
\end{align*}
Then one computes that $$\det (\frac{\partial P_i}{\partial z_j})_{i,j \in I}= \pm \g_{k_0}^{k_0+1} \prod_{i \in J}\bc_i\i \neq 0,$$ which means there are only finitely many solutions to the equations $P_i=0$ with $i \in I$.

Now we assume $\g_k=0$ for all $k \in \ZZ_{\ge 0}$ and (\ref{key2'''}) becomes
\begin{align}\label{e2}
P_{i_0}=\sum_{i \in I}\b_i z_i+\sum_{k=0}^\infty \a_k (\bc_{i_0}\i (z_{i_0}^q-z_{i_0}))^{q^k}=0.
\end{align}

We are reduced to show there are only finitely many solutions to (\ref{key2'}), (\ref{key2''}) and (\ref{e2}). Note that $\bc_i \in \kk$ with $i \in J$ are linear independent over $\FF_q$ by our choice of $E$ and that not all $\b_i$ for $i \in I$ is zero. The statement now follows directly from Lemma \ref{teq} below.
\end{proof}

\begin{lemma}\label{teq}
Let $I$ be finite index set and $J \subset I$ a nonempty subset. Let $C \subset \kk[z]$ be the $k$-vector space spanned by $z^{q^k}$ with $k \in \ZZ_{\ge 0}$. Let $T_I$ ($T_{I \setminus J}$) be the $\kk$-vector subspace of $\kk[z_i; i \in I]$ spanned by $z_i$ with $i \in I$ ($i \in I \setminus J$). We fix $i_0 \in J$.

Let $\bc_i \in \kk$ with $i \in J$ be linear independent over $\FF_q$. Let $t \in T_I$ and $g \in C$ such that either $t \notin T_{I \setminus J}$ or $g$ is nontrivial. Then there are only finitely many solutions to the following equations in variables $z_i$ for $i \in I$.
\begin{align*}
P_i&=\bc_i\i(z_i^q-z_i)-\bc_{i_0}\i(z_{i_0}^q-z_{i_0})=0, \text{ if } i \in J \setminus \{i_0\};\\
P_i&=z_i^q-z_i=0, \text{ if } i\in I \setminus J;\\
P_{t,g}&=t+g(\bc_{i_0}\i(z_{i_0}^q-z_{i_0}))=0.
\end{align*}
\end{lemma}

\begin{proof}
Assume $t=\sum_{i \in I} \b_i z_i$ and $g=\sum_{k=0}^\infty \a_k z^{q^k}$, where $\b_i, \a_k \in \kk$ and $\a_k=0$ for all but finitely many $k \in \ZZ_{\ge 0}$. We argue by induction on $|J|$. If $|J|=1$, then $J=\{i_0\}$ and $z_i^q-z_i=0$ for $i \in I \setminus \{i_0\}$. By $P_{t,g}=0$ we have
\begin{align*}
\b_{i_0} z_{i_0} + g(\bc_{i_0}\i(z_{i_0}^q-z_{i_0})) \in \sum_{i \in I \setminus \{i_0\}}\FF_q \b_i.
\end{align*}
By our assumption, either $\b_{i_0} \neq 0$ or $g$ is nontrivial, then $\b_{i_0} z_{i_0} + g(\bc_{i_0}\i(z_{i_0}^q-z_{i_0}))$ is not a constant polynomial and hence there are only finitely many solutions. Now we assume the statement holds when $1 \le |J|<s$, we show that it also holds when $|J|=s$. Let $J_1=\{i \in J; \b_i \neq 0\}$. Then $J_1 \neq \emptyset$ because $t \in T_{I \setminus J}$. If $J_1 \subsetneq J$, the statement follows from induction hypothesis by replacing $I$ and $J$ with $J_1 \cup (I \setminus J)$ and $J_1$ respectively. So we assume that $\b_i \neq 0$ for all $i \in J$. Let $P_{i_0}=P_{t,g}$. One computes that $$\pm \det (\frac{\partial P_i}{\partial z_j})_{i,j \in I}=\b_{i_0}-\bc_{i_0}\i \a_0+\sum_{i \in J \setminus \{i_0\}} \frac{\bc_i}{\bc_{i_0}}\b_j.$$ Then we can assume that $\det (\frac{\partial P_i}{\partial z_j})_{i,j \in I}=0$, that is, \begin{align}\label{e1} \a_0-\sum_{i \in J}\bc_i \b_i=0.\end{align}

By $(\frac{1}{\b_{i_0}} P_{i_0})^q-\frac{1}{\b_{i_0}} P_{i_0}=0$ we have
\begin{align*}
(z_{i_0}^q-z_{i_0})+\sum_{i \in I \setminus \{i_0\}}\frac{\b_i}{\b_{i_0}}(z_i^q-z_i)+\sum_{i \in I \setminus \{i_0\}}((\frac{\b_i}{\b_{i_0}})^q-\frac{\b_i}{\b_{i_0}}) z_i^q+(\frac{g(r)}{\b_{i_0}})^q-\frac{g(r)}{\b_{i_0}}=0,
\end{align*}
where $r=\bc_{i_0}\i(z_{i_0}^q-z_{i_0})$. By (\ref{e1}) and the relations $z_i^q-z_i=\bc_i \bc_{i_0}\i (z_{i_0}^q-z_{i_0})$ for $i \in I$ we have
\begin{align*}
\sum_{i \in I \setminus \{i_0\}}((\frac{\b_i}{\b_{i_0}})^q-\frac{1}{\b_{i_0}}) z_i^q + (\frac{g(r)}{\b_{i_0}})^q-\frac{1}{\b_{i_0}}(g-\a_0 z)(r)=0.
\end{align*}
Note that $\frac{1}{\b_{i_0}}(g-\a_0 z) \in C$ has a $q$-th root $g_1' \in C$. Let $g'=\frac{1}{\b_{i_0}}g-g_1' \in C$ and $t'=\sum_{i \in I \setminus \{i_0\}}\b_i' z_i \in T_I$, where $\b_{i}'$ is the $q$-th root of $(\frac{\b_i}{\b_{i_0}})^q-\frac{\b_i}{\b_{i_0}}$. Let $I'=I \setminus \{i_0\}$ and $J'=J \setminus \{i_0\}$. Since $|J|=s \ge 2$, $J' \neq \emptyset$. We claim that either $t' \notin T_{I' \setminus J'}$ or $g'$ is nontrivial. Assume otherwise. Then $g'$ is a zero polynomial, hence so is $g$, which implies $\a_0=0$. On the other hand, since $t' \in T_{I \setminus J}$ we have $\b_i'^q=(\frac{\b_i}{\b_{i_0}})^q-\frac{\b_i}{\b_{i_0}}=0$ for $i \in J \setminus \{i_0\}$. Hence $\frac{\b_i}{\b_{i_0}} \in \FF_q$ for all $i \in J$. Therefore by (\ref{e1}) we have $$\sum_{i \in J} \frac{\b_i}{\b_{i_0}} \bc_i=0,$$ which contradicts our assumption that $\bc_i$ with $i \in J$ are linear independent over $\FF_q$. The claim is proved.

Choose $i_0' \in J'$. Let $P_i'=\bc_i\i (z_i^q-z_i)-\bc_{i_0'}\i(z_{i_0'}^q-z_{i_0'})$ if $i \in J' \setminus \{i_0'\}$ and $P_i'=z_i^q-z_i$ if $i \in I' \setminus J'$. Let $P_{t',g'}'=t'+g'(\bc_{i_0'}\i(z_{i_0'}^q-z_{i_0'}))$. Note that $\bc_i$ with $i \in J'$ are linear independent over $\FF_q$. Applying induction hypothesis, we see that there are only finitely many solutions to $P_{t',g'}'$ and $P_i'$ with $i \in I'$. Since $\bc_{i_0}\i(z_{i_0}^q-z_{i_0})=\bc_{i_0'}\i(z_{i_0'}^q-z_{i_0'})$, hence there are only finitely many solutions to $P_{t, g}$ and $P_i$ with $i \in I$. The proof is finished.
\end{proof}

\end{document}